\documentclass[12pt,amstex]{amsart}

\usepackage{mathptmx}%×ÖÌå
\usepackage{mathrsfs}

\usepackage{verbatim}
\usepackage{url}
\usepackage[all]{xy}
\usepackage{color}

\usepackage[colorlinks=true,citecolor=blue]{hyperref}%ÒýÓÃºÍ¶¨ÀíÑÕÉ«

\usepackage{stmaryrd}
\usepackage{epsfig}
\usepackage{amsmath}
\usepackage{amssymb}
\usepackage{appendix}
\usepackage{amscd}
\usepackage{graphicx}
\usepackage{pstricks}

\theoremstyle{plain}
\newtheorem{thm}{Theorem}[section]
\newtheorem{lem}[thm]{Lemma}

\newtheorem{prop}[thm]{Proposition}

\theoremstyle{definition}
\newtheorem{de}[thm]{Definition}

\theoremstyle{remark}
\newtheorem{rem}[thm]{Remark}

\def \N {\mathbb N}

\def \A {\mathcal A}

\def \a {\alpha }

\def \ep {\epsilon}

\begin{document}
\title{Feldman-Katok pseudo-orbits and topological pressure}

\author{Fangzhou Cai}

\address[F. Cai]{School of Mathematics and Systems Science, Guangdong Polytechnic Normal University, Guangzhou, 510665, PR China}
\email{cfz@mail.ustc.edu.cn}

\maketitle

\begin{abstract}
In this paper we introduce the notion of Feldman-Katok pseudo-orbits and use it to study topological pressure. We prove that the topological pressure of a dynamical system can be computed  by measuring  the Feldman-Katok pseudo-orbits complexity of the shift map on the potential function, which extends Barge and  Swanson's result on entropy and pseudo-orbits \cite{pse2}.
\end{abstract}

\section{Introduction}

By a {\it topological dynamical system} (TDS for short) we mean a pair $(X,T)$ where $X$ is a compact metric space with metric $d$ and $T:X\to X$ is a continuous  map.
Topological entropy was introduced originally by Adler, Konheim and McAndrew  \cite{entro} in 1965. Later, Dinaburg \cite{entro1} and Bowen \cite{entro2} gave several equivalent
definitions by using separated and spanning sets. Bowen also gave a
characterization of dimension type for topological entropy, which  generalized the definition  on non-compact sets \cite{6}. Topological entropy is an important dynamical invariant of topological conjugacy. Roughly speaking,  Topological entropy measures the maximal exponential growth rate of different forward orbits for a topological dynamical system.

In the years since their development by Bowen \cite{bo} and Conley \cite{co}, pseudo-orbits have
proved to be a powerful tool in dynamical systems.
A remarkable result by Misiurewicz
\cite{pse1} stated that the topological entropy can be computed by measuring the exponential growth rate of the
numbers of pseudo-orbits. Later Barge and  Swanson \cite{pse2} gave a similar result for periodic pseudo-orbits. In \cite{pse3}, Hurley
considered pseudo-orbits for inverse images and showed that the point entropy of pseudo-orbits is in fact equal to the topological entropy.
 In \cite{yz}, Yan and Zeng obtained similar results for continuous maps on compact uniform spaces. Recently, Cheng and Li used pseudo-orbits to characterize the notion of scaled pressure \cite{cl}.

%It is a natural question to  distinguish systems with infinite
% topological entropy.
%Topological mean  dimension, as a dynamical quantity,  has been used to analyze  systems with infinite entropy. This
% invariant
%  was first introduced by Gromov \cite{gro}, and and further studied by  Lindenstrauss and  Weiss in \cite{lin2}.

As a non-trivial and natural generalization of topological entropy, the notion of topological pressure was introduced  in  dynamical systems by Ruelle \cite{ru}   inspired by the theory of Gibbs states in statistical mechanics. Later Walters  \cite{wal} extended it to general topological dynamical systems. From a viewpoint of dimension theory, Pesin and Pitskel \cite{31} generalized Bowen’s definition of
topological entropy and defined the
topological pressure on non-compact sets.
Topological pressure is a fundamental notion in thermodynamic
formalism  and constitute the main components
of the thermodynamic formalism (\cite{pre1,pre2,pre3,pre4}).

\medskip

In 2017, Kwietniak and \L acka \cite{kl} introduced the Feldman-Katok
metric  as the topological counterpart of edit
distance $\bar{f}$ which was introduced by Feldman \cite{fk1} to study loosely Bernoulli system (see also \cite{fk2,fk3}). In recent years, the  Feldman-Katok
metric proved to be a useful tool in dynamical systems. In \cite{gk}, the authors  used it to characterize zero
entropy loosely Bernoulli systems and presented a
purely topological characterization of their topological models. In \cite{fk4}, Downarowicz, Kwietniak and \L acka introduced
the idea of $\bar{f}$-pseudometric to  finite-valued
stationary stochastic processes and used it to study the entropy rate. In \cite{cai},
Cai and Li  gave
entropy formulas defined by Feldman-Katok metric.
%Inspired by the definitions of topological mean  dimension and pressure,   Chen and Li \cite{cl} introduced a notion of pressure-like complexity functions, which is called scaled pressure.  Scaled pressure with respect to
%pseudo-orbits was also introduced and  the relations of these notions were studied in \cite{cl}.

% In \cite{xcy}, the authors introduced  mean dimension for Feldman-Katok metric and studied its relationship with the classical metric mean dimension.
 
 The advantage to use Feldman-Katok metric  is that  it allows time delay by ignoring the synchronization of points in orbits with only order preserving required. Inspired by previous nice works, in this paper, we   bring the  idea  from the definition of Feldman-Katok metric into pseudo-orbits and attempt to use it to study topological pressure. 
 
 The paper is organized as follows: We first introduce  the notion of   Feldman-Katok pseudo-orbits. Then we give a formula for topological pressure with respect to Feldman-Katok pseudo-orbits.  We also consider scaled pressure and show that in \cite[Theorem F]{cl}, the    Lipschitz continuity of $T$ is redundant.
 Finally we introduce and study topological pressure with respect to Feldman-Katok metric.

\section{ Feldman-Katok pseudo-orbits and topological pressure}

In this section we introduce the notion of Feldman-Katok pseudo-orbits and use it to study topological pressure. Our idea is  from the definition of Feldman-Katok metric,  which allows time delay  in orbits with only order preserving required.

\subsection{Feldman-Katok pseudo-orbits}
Let $(X,T)$ be a TDS
and $\a>0$. Recall that we say $(x_0,x_1,\ldots)\in X^\N$ is an {\it$\a$-pseudo-orbit} if $$d(Tx_i,x_{i+1})\leq\a, i=0,1\ldots,n\ldots.$$ 
Denote by $PO_\a(X,T)$ or simply $PO_\a$ the set of $\a$-pseudo-orbits.

We bring time delay and ``jump" in classical pseudo-orbits and give the definition of  Feldman-Katok pseudo-orbits.
\begin{de}
	Let $(X,T)$ be a TDS, $n\in\N$ and $\a,\delta>0.$
We say $(x_0,x_1,\ldots,x_{n-1})\in X^n$ is a {\it $FK$-$\alpha$-pseudo-chain of density $1-\delta$} if there exist $(1-\delta)n<k\leq n$ and $$0\leq i_1<i_2<\ldots<i_k\leq n-1,$$ $$ 0\leq j_1<j_2<\ldots<j_k\leq n-1$$ such that $$d(T^{j_{t+1}-j_t}(x_{i_t}),x_{i_{t+1}})\leq\alpha,t=1,\ldots,k-1.$$ 
\end{de}
\begin{rem}
	In the definition above, if we choose $k=n$ and $i_1=j_1=0,\ldots,i_n=j_n=n-1$, we get the  definition of   $\alpha$-pseudo-chain (see \cite{pse2}).
\end{rem}
\begin{de}
Let $(X,T)$ be a TDS,
 $\delta>0$ and $N_\delta\in \N$. We say $(x_0,x_1,\ldots)\in X^\N$ is a {\it$FK$-$\alpha$-pseudo-orbit of density $1-\delta$  controlled by $N_\delta$} if there exist a sequence  $ 0\leq s_0<s_1<\ldots<s_i<\ldots$ in $\N$ with  $$s_0\leq N_\delta, s_{i+1}-s_i\leq\ N_\delta,i=0,1,\ldots,$$ such that for all   $m>n$, $(x_{s_n},x_{s_n+1},\ldots,x_{s_m-1})$ is a $FK$-$\alpha$-pseudo-chain of density $1-\delta$. 
\end{de}Denote  by $FKPO_{\alpha,\delta}(X,T, N_\delta)$ the set of $FK$-$\alpha$-pseudo-orbit of density $1-\delta$ controlled by $N_\delta$.
\begin{rem}
	\begin{enumerate}
		\item $N_\delta$ is needed to control the length, for we must avoid the situation that  the first finite components of a point ${\bf x}$ can be chosen arbitrary, since  in this situation we  obtain a dense set in $X^\N$.
		\item  It is easy to see that $FKPO_{\alpha,\delta}(X,T,N_\delta)$ contains $PO_\a(X,T)$	 by  definition.
		\item   We can see points with the following form $$(x,y_1,T^{2}x,y_3,\ldots,T^{2n}x,y_{2n+1},\ldots)$$ are in $ FKPO_{\alpha,\delta}(X,T,N_\delta)$ for $\a>0,\delta>\frac{1}{2}$ and $N_\delta\geq 4.$ Hence there is an essential difference between $FK$-pseudo-orbits and   pseudo-orbits.
		\item  If $\delta N_\delta<1$, the choice of $k$  in the definition is only $n$, hence to avoid triviality, we can make $\delta N_\delta$ large enough by choosing appropriate $N_\delta$.
		\item It is easy to see that  $FKPO_{\alpha,\delta}(X,T,N_\delta)$ is monotonically increasing with respect to $\a$. But $FKPO_{\alpha,\delta}(X,T,N_\delta)$ may not monotonous  with respect to $\delta$.
	\end{enumerate}
\end{rem}

\medskip

In the sequel we fix $\{N_\delta: \delta>0\}\subset \N$.

 We simply write $FKPO_{\alpha,\delta}$ instead of $FKPO_{\alpha,\delta}(X,T, N_\delta)$.	

\medskip

The following proposition shows that $FKPO_{\alpha,\delta}$ has a good structure.
\begin{prop}
	Let $(X,T)$ be a TDS and $\alpha,\delta>0$. Then $FKPO_{\alpha,\delta}$ is closed in $X^\N$ and invariant under the shift $\sigma$.
\end{prop}
\begin{proof}
Since there is no restriction on $(x_0,\ldots,x_{s_0-1})$ in the definition of $FKPO_{\alpha,\delta}$,	it is easy to see that  $FKPO_{\alpha,\delta}$ is   invariant under $\sigma$.
		Now we prove $FKPO_{\alpha,\delta}$ is closed. 
		
	Let ${\bf x}^h\in FKPO_{\alpha,\delta}$ with ${\bf x}^h\rightarrow {\bf x}$.
	By definition there exist  a sequence $s_0^h<s_1^h<\ldots<s_i^h<\ldots$  with $$s_0^h\leq\ N_\delta, s_{i+1}^h-s_i^h\leq N_\delta,   \ i=0,1,\ldots,$$ such that for all   $m>n$, $(x_{s_n^h}^h,x_{s_n^h+1}^h,\ldots,x_{s_m^h-1}^h)$ is a $FK$-$\alpha$-pseudo-chain of density $1-\delta$. We can find  $ s_0<s_1<\ldots<s_i<\ldots$ with $$s_0\leq N_\delta, s_{i+1}-s_i\leq N_\delta,i=0,1,\ldots,$$ such that for any $l\in\N$, there exist infinite $h$ satisfy $$(s_0^h,s_1^h,\ldots,s_l^h)=(s_0,s_1,\ldots,s_l).$$ Fix  $m>n$. Now we prove $(x_{s_n},x_{s_n+1},\ldots,x_{s_m-1})$ is a $FK$-$\alpha$-pseudo-chain of density $1-\delta$.  For $m$, there exist infinite $h$ with $$(s_0^{h},s_1^{h},\ldots,s_m^{h})=(s_0,s_1,\ldots,s_m).$$ For such $h$, since  $$(x_{s_n^{h}}^{h},x_{s_n^{h}+1}^{h},\ldots,x_{s_m^{h}-1}^{h})=(x_{s_n}^{h},x_{s_n+1}^{h},\ldots,x_{s_m-1}^{h})$$ is a $FK$-$\alpha$-pseudo-chain of density $1-\delta$, there exist  $k_h$ with $$(1-\delta)(s_m-s_n)<k_h\leq (s_m-s_n)$$ and $$0\leq i_1^h<i_2^h<\ldots<i_{k_h}^h\leq s_m-s_n-1,$$ $$0\leq j_1^h<j_2^h<\ldots<j_{k_h}^h\leq s_m-s_n-1,$$   such that  $$d(T^{j_{t+1}^h-j_t^h}x_{i_t^h+s_n}^h,x_{i_{t+1}^h+s_n}^h)\leq\alpha, t=1,\ldots,k_h-1.$$   We can choose some  $k$ with $$(1-\delta)(s_m-s_n)<k\leq (s_m-s_n)$$  and $$0\leq i_1<i_2<\ldots<i_k\leq  s_m-s_n-1,$$ $$0\leq j_1<j_2<\ldots<j_k\leq  s_m-s_n-1,$$ such that there exist infinite $h$ satisfy $k_h=k$ and $$ (i_1^h,i_2^h,\ldots,i_{k}^h)=(i_1,i_2,\ldots,i_{k}),$$ $$(j_1^h,j_2^h,\ldots,j_{k}^h)=(j_1,j_2,\ldots,j_{k}).$$ For such $h$ we have $$d(T^{j_{t+1}-j_t}x_{i_t+s_n}^h,x_{i_{t+1}+s_n}^h)\leq\alpha, t=1,\ldots,k-1.$$ Let $h\to \infty$ we have  $$d(T^{j_{t+1}-j_t}x_{i_t+s_n},x_{i_{t+1}+s_n})\leq\alpha, t=1,\ldots,k-1.$$ Hence $(x_{s_n},x_{s_n+1},\ldots,x_{s_m-1})$ is a $FK$-$\alpha$-pseudo-chain of density $1-\delta$.
\end{proof}

\subsection{Topological pressure with respect to Feldman-Katok pseudo-orbits}
In this subsection we use   Feldman-Katok pseudo-orbits  to study topological pressure.
 
 \medskip
 
 First let us recall some basic concepts.

Let $(X,T)$ be a TDS. For $x,y\in X,$  $n\in\N$ and $\ep>0$, set $$d_n(x,y)=\max_{0\leq i\leq n-1}d(T^ix,T^iy)$$
be the Bowen metric.
We say $E\subset X$ is  an {\it$(n,\ep)$-separated set of $X$}  if for each $ x\neq  y\in E$,  $d_n(x,y)>\ep$.
We say $F\subset X$  is an {\it $(n,\ep)$-spanning set  of  $X$} if for every $x\in X$ there is $y\in F$ with $d_n(x,y)\leq\ep.$

Let  $f\in C(X)$, define $$sp_{n,\ep}(X,T,f)=\inf\{\sum_{x\in F}e^{\sum_{i=0}^{n-1}f(T^ix)}: F \text{ is\ an }(n,\ep) \text{-spanning\ set \ of } X\},$$
$$sr_{n,\ep}(X,T,f)=\sup\{\sum_{x\in E}e^{\sum_{i=0}^{n-1}f(T^ix)}: E \text{ is\ an }(n,\ep) \text{-separated\ set \ of } X\}.$$
It is known that the topological  pressure of $(X,T)$ with potential $f$ is
	defined by
\begin{equation*}
	\begin{split}
		P(X,T,f)&=\lim_{\ep\to0}\varlimsup\limits_{n\to\infty}\frac{1}{n}\log sp_{n,\ep}(X,T,f)\\
		&=\lim_{\ep\to0}\varlimsup\limits_{n\to\infty}\frac{1}{n}\log sr_{n,\ep}(X,T,f).	
	\end{split}
\end{equation*}

\bigskip

 For ${\bf x}=(x_0,x_1,\ldots),{\bf y}=(y_0,y_1,\ldots)$ in $X^\N,$ the metric $\tilde{d}$ on $X^\N$ is defined by  $$\tilde{d}({\bf x},{\bf y})=\sum_{i=0}^{\infty}\frac{d(x_i,y_i)}{2^i}.$$  
  For $Y\subset X^\N$, we say $E\subset Y$ is  an {\it$(n,\ep)$-separated set of $Y$} if for each ${\bf x}\neq {\bf y}$ in $E$, there exists $0\leq i\leq n-1$ such that $d(x_i,y_i)>\ep$.
 We say $F\subset Y$  is an {\it $(n,\ep)$-spanning set  of  $Y$} if for every ${\bf x}\in Y,$ there is ${\bf y}\in F$ with $d(x_i,y_i)<\ep,0\leq i\leq n-1.$
 
For $\alpha,\delta>0$ and $f\in C(X)$, define \begin{equation*}
\begin{split}
&FKPOsp_{n,\ep; \alpha,\delta}(X,T,f)=\\
&\inf\{\sum_{{\bf x}\in F}e^{\sum_{i=0}^{n-1}f(x_i)}: F \text{ is\ an }(n,\ep) \text{-spanning\ set \ of } FKPO_{\alpha,\delta}\},
\end{split}
\end{equation*}
\begin{equation*}
\begin{split}
 &FKPOsr_{n,\ep;\alpha,\delta}(X,T,f)=\\
 &\sup\{\sum_{{\bf x}\in E}e^{\sum_{i=0}^{n-1}f(x_i)}: E \text{ is\ an }(n,\ep) \text{-separated\ set \ of } FKPO_{\alpha,\delta}\}.
\end{split}
\end{equation*}
 We simply write $FKPOsp_{n,\ep;\alpha,\delta}$ and $FKPOsr_{n,\ep;\alpha,\delta}$ if there is no confusion.
 
 \medskip
 
Similar as the definition in \cite{pse2}, we
	define \begin{equation*}
		\begin{split}
			PFKPO(X,T,f,\ep,\delta)&=\inf_{\alpha> 0}\varlimsup\limits_{n\to\infty}\frac{1}{n}\log FKPOsp_{n,\ep;\alpha,\delta}(X,T,f)\\&=\lim_{\alpha\to 0}\varlimsup\limits_{n\to\infty}\frac{1}{n}\log FKPOsp_{n,\ep;\alpha,\delta}(X,T,f).
		\end{split}
	\end{equation*}
The limit exists since $FKPO_{\alpha,\delta}$ is monotonically increasing with respect to $\a$. Now we define
	$$PFKPO(X,T,f,\ep)=\varlimsup_{\delta\to0}	PFKPO(X,T,f,\ep,\delta).$$ 
Note that here $FKPO_{\alpha,\delta}$ may not monotonous  with respect to $\delta$, hence ``$\varlimsup\limits_{\delta\to0}$" can not be replaced by ``$\inf\limits_{\delta> 0}$" or ``$\lim\limits_{\delta\to0}$".

 It is easy to see $PFKPO(X,T,f,\ep)$ is non-increasing with respect to $\ep$. We define
$$PFKPO(X,T,f)=\lim_{\ep\to 0}PFKPO(X,T,f,\ep)$$
 be the {\it topological pressure with respect to Feldman-Katok pseudo-orbits.}

\medskip

The next lemma shows that we can use separated set in the definition above.
 \begin{lem}\label{lem}
 	Let $(X,T)$ be a TDS and $f\in C(X).$ Then
$$PFKPO(X,T,f)=\lim_{\ep\to 0}\varlimsup_{\delta\to 0}\inf_{\alpha>0}\varlimsup\limits_{n\to\infty}\frac{1}{n}\log FKPOsr_{n,\ep;\alpha,\delta}(X,T,f).$$
	\end{lem}
\begin{proof}
	By definition it is easy to see $$PFKPO(X,T,f)\leq\lim_{\ep\to 0}\varlimsup_{\delta\to 0}\inf_{\alpha>0}\varlimsup\limits_{n\to\infty}\frac{1}{n}\log FKPOsr_{n,\ep;\alpha,\delta}.$$ Now we prove the opposite direction.
 Let $\alpha,\delta>0.$ For  $\ep>0,$ set $$L(\ep)=\sup_{|x-y|\leq\ep}|f(x)-f(y)|.$$Assume that $E$ be an $(n,\ep)$-separated set and $F$ be an $(n,\frac{\ep}{2})$-spanning set of $FKPO_{\alpha,\delta}$. The map $\phi:E\to F$ is defined by choosing for each ${\bf x}\in E,$ some point $\phi({\bf x})\in F$ with $$d(x_i,\phi({\bf x})_i)\leq\frac{\ep}{2},i=0,\ldots,n-1.$$ Then $\phi$ is injective and
\begin{equation*}
\begin{split}
\sum_{{\bf y}\in F}e^{\sum_{i=0}^{n-1}f(y_i)}&\geq \sum_{{\bf x}\in E}e^{\sum_{i=0}^{n-1}f(\phi({\bf x})_i)}\\
&=\sum_{{\bf x}\in E}e^{\sum_{i=0}^{n-1}f(x_i)}e^{\sum_{i=0}^{n-1}f(\phi({\bf x})_i)-\sum_{i=0}^{n-1}f(x_i)}\\
&\geq e^{-nL(\frac{\ep}{2})}\sum_{{\bf x}\in E} e^{\sum_{i=0}^{n-1}f(x_i)}.
\end{split}
\end{equation*}

Thus
\begin{equation*}
\begin{split}
&\varlimsup_{\delta\to 0}\inf_{\alpha>0}\varlimsup\limits_{n\to\infty}\frac{1}{n}\log FKPOsr_{n,\ep;\alpha,\delta}\\
&\leq\varlimsup_{\delta\to 0}\inf_{\alpha>0}\varlimsup\limits_{n\to\infty}\frac{1}{n}\log FKPOsp_{n,\frac{\ep}{2};\alpha,\delta}+L(\frac{\ep}{2}).
\end{split}
\end{equation*}
Let $\ep\to 0$, we have 
$$\lim_{\ep\to 0}\varlimsup_{\delta\to 0}\inf_{\alpha>0}\varlimsup\limits_{n\to\infty}\frac{1}{n}\log FKPOsr_{n,\ep;\alpha,\delta}\leq PFKPO(X,T,f).$$
\end{proof}

 Now we show that the topological pressure with  respect to Feldman-Katok pseudo-orbits is equal to the classical topological pressure. We have following theorem:
 \begin{thm}\label{main}
 Let $(X,T)$ be a TDS and $f\in C(X)$. Then 	$$PFKPO(X,T,f)=P(X,T,f).$$
 \end{thm}
\begin{proof}
	By definition we have $PFKPO(X,T,f)\geq P(X,T,f).$ Now we prove the opposite direction.
	
	Fix $\ep>0.$	Let $X=\bigcup_{i=1}^{l} X_i$ with $diam X_i<\frac{\ep}{4},i=1,\ldots,l.$ Choose some $z_i\in X_i$.

	For $\delta<\frac{\ep}{2}, n>\max\{\frac{N_{\delta}}{\delta},\frac{2diam X}{\delta}\}, m>3n,$ let $F$ be an $(n+m,\frac{\delta}{4})$-spanning set of $X$. We can find $$0<\delta_0<\delta_1<\ldots<\delta_{2(n+m)}=\frac{\delta}{4}$$ such that for every $\delta_k, k=0,1,\ldots,2(n+m)-1$,
	$$d(x,y)<\delta_k\Rightarrow d(T^ix,T^iy)<\delta_{k+1}, i=0,\ldots,2(n+m)-1.$$
	
	Let $\a<\min\{\delta_0, \delta_{k+1}-\delta_k:k=0,1,\ldots,2(n+m)-1 \}.$ 
	
For $x\in F,$ let  $F_x$ be the collection of points with following form:	
	$$(y_0,y_1,\ldots,y_{i_1-1},T^{j_1}x,y_{i_1+1},\ldots,y_{i_k-1},T^{j_k}x,y_{i_k+1},\ldots, y_{m+n-1},z_1,z_1,\ldots),$$
	where $$(1-3\delta)(n+m)<k\leq (n+m),$$ $$0\leq i_1<\ldots<i_{k}\leq n+m-1,$$ $$0\leq j_1<\ldots<j_{k}\leq n+m-1$$
 and $y_i\in\{z_1,\ldots,z_l\}.$
	
	We have $$|F_x| \leq\sum_{(1-3\delta)(n+m)<k\leq n+m}(C_{n+m}^k)^2l^{n+m-k}.$$

	{\bf Claim}:  $$FKPO_{\alpha,\delta}\subset\bigcup_{x\in F}\bigcup_{{\bf y}\in F_x}B_{\tilde{d}_m}({\bf y},\ep),$$
	where $B_{\tilde{d}_m}({\bf y},\ep)$ denotes the open  $\ep$-ball of ${\bf y}$ with respect to $\tilde{d}_m$.
	
	proof of claim:
	For ${\bf x}\in FKPO_{\alpha,\delta},$
	by definition there exist $0<s_0<s_1<\ldots<s_i<\ldots$ with $$s_0\leq N_{\delta}, s_{i+1}-s_i\leq N_{\delta},i=0,1,\ldots,$$ such that for all   $b>a$, $(x_{s_a},x_{s_a+1},\ldots,x_{s_b-1})$ is a $FK$-$\alpha$-pseudo-chain of density $1-\delta$.
	Let $b$ satisfy $s_b\leq n+m$ and $s_{b+1}>m+n$. It is easy to check $$s_b-s_0>(1-2\delta)(n+m).$$ Since $(x_{s_0},x_{s_0+1},\ldots,x_{s_b-1})$ is a $FK$-$\alpha$-pseudo chain of density $1-\delta$, we can find $k$ with $$(1-3\delta)(n+m)\leq(1-\delta)(s_b-s_0)<k\leq s_b-s_0\leq n+m$$ and
	$$s_0\leq i_1<i_2<\ldots<i_k\leq s_b-1\leq n+m-1,$$ $$0\leq j_1<j_2<\ldots<j_k\leq s_b-s_0-1\leq n+m-1$$  such that $$d(T^{j_{t+1}-j_t}(x_{i_t}),x_{i_{t+1}})\leq\alpha,t=1,\ldots,k-1.$$
	Now we prove 
	\begin{equation}\label{1}
	d(T^{j_{t}-j_1}(x_{i_1}),x_{i_{t}})<\delta_{2(t-2)}<\frac{\delta}{4},t=2,\ldots,k.
	\end{equation}
	We prove by induction. When $t=2$ it is clear. If (\ref{1}) holds for $t$, we have 
	$$d(T^{j_{t+1}-j_1}(x_{i_1}),T^{j_{t+1}-j_t}x_{i_{t}})<\delta_{2(t-2)+1},$$
	hence $$d(T^{j_{t+1}-j_1}(x_{i_1}),x_{i_{t+1}})<\delta_{2(t-2)+1}+\a\leq\delta_{2(t-1)}.$$
	By induction (\ref{1}) is proved.

Since	$F$ is an $(n+m,\frac{\delta}{4})$-spanning set of $X$, there exists $x\in F$ such that $$d(T^hx,T^hx_{i_1})<\frac{\delta}{4}, h=0,\ldots,n+m-1.$$
	Hence $$d(T^{j_{t}-j_1}x,x_{i_{t}})<\frac{\delta}{2},t=1,\ldots,k.$$
	We assume $x_i\in X_{h_i}, i=0,1,\ldots,m+n-1.$
Note that $\frac{1}{2^{n-1}}<\frac{\ep}{2diam X}$. We have  ${\bf x}\in B_{\tilde{d}_m}({\bf y},\ep),$
	where $${\bf y}=(z_{h_0},z_{h_1},\ldots,z_{h_{i_1-1}},x,z_{i_1+1},\ldots,z_{h{_{i_k-1}}},T^{j_k-j_1}x,z_{h_{i_k+1}},\ldots, z_{h_{m+n-1}},z_1,z_1,\ldots).$$

	The claim is proved.
	
Denote
 \begin{equation*}
	\begin{split}
	&Cov_{n,\ep}(\beta,\delta)=\\
	&\inf\{\sum_{A\in \mathcal{A}}\sup_{{\bf x}\in A}e^{\sum_{i=0}^{n-1}f(x_i)}:\mathcal{A}\text{ is\ a\ finite\ subcover\ of\ } \vee_{i=0}^{n-1}\sigma^{-i}\mathcal{A}_\ep \text{ covers } FKPO_{\beta,\delta} \},
	\end{split}
	\end{equation*}
	where  $\A_\ep$ denote the  cover of $(X^\N,\tilde{d})$  consists of all open $\ep$-balls.  
	
	By claim we have $$Cov_{m,\ep}(\alpha,\delta)\leq \sum_{x\in F}\sum_{{\bf y}\in F_x}\sup_{{\bf z}\in B_{\tilde{d}_m}({\bf y},\ep)}e^{\sum_{i=0}^{m-1}f(z_i)}.$$
Note that $\alpha$ dependents on $m$ and $\delta$.
	
	For $x\in F$, if ${\bf y}\in F_x$ and ${\bf z}\in B_{\tilde{d}_m}({\bf y},\ep)$ we have $d(y_i,z_i)<\ep,i=0,\ldots,m-1.$
Note that $m>3n$, we have$$(1-3\delta)(n+m)-n>(1-4\delta)m$$ and $$\sum_{i=0}^{m-1}f(z_i)-\sum_{i=0}^{m-1}f(T^ix)\leq mL(\ep)+2||f||4\delta m,$$
where $$L(\ep)=\sup_{|x-y|\leq\ep}|f(x)-f(y)|.$$ Hence
\begin{equation*}
\begin{split}
\sum_{{\bf y}\in F_x}\sup_{{\bf z}\in B_{\tilde{d}_m}({\bf y},\ep)}e^{\sum_{i=0}^{m-1}f(z_i)}
\leq |F_x|e^{\sum_{i=0}^{m-1}f(T^ix)}e^{(mL(\ep)+8||f||\delta m)}.
\end{split}
\end{equation*}
Note that
\begin{equation*}
\begin{split}
|F_x| &\leq\sum_{(1-3\delta)(n+m)<k\leq n+m}(C_{n+m}^k)^2l^{n+m-k}\\
&\leq 3\delta(n+m)(C_{n+m}^{[3\delta(n+m)]})^2l^{3\delta(n+m)}.
\end{split}
\end{equation*}
We have
\begin{equation*}
\begin{split}
Cov_{m,\ep}(\alpha,\delta)&\leq 3\delta(n+m)(C_{n+m}^{[3\delta(n+m)]})^2l^{3\delta(n+m)}\\
&\cdot e^{(mL(\ep)+8||f||\delta m)}\sum_{x\in F}e^{\sum_{i=0}^{m-1}f(T^ix)}\\ &\leq3\delta(n+m)(C_{n+m}^{[3\delta(n+m)]})^2l^{3\delta(n+m)}\sum_{x\in F}e^{\sum_{i=0}^{m+n-1}f(T^ix)}\\ &\cdot e^{(mL(\ep)+8||f||\delta m)+n||f||}.
\end{split}
\end{equation*}
Hence 
\begin{equation*}
\begin{split}
&\inf_m \inf_{\alpha>0}\frac{1}{m}\log Cov_{m,\ep}(\alpha,\delta)\\
&\leq (L(\ep)+8||f||\delta )
+\lim_{m \to \infty}\frac{1}{m}\log (C_{n+m}^{[3\delta(n+m)]})^2+3\delta\log l+\varlimsup_{m\to \infty}\frac{1}{m}\log sp_{m,\frac{\delta}{4}}.
\end{split}
\end{equation*}
It is easy to see that for all $m,n\in\N$ and $\beta>0$, we have 
$$Cov_{m+n,\ep}(\beta,\delta)\leq Cov_{m,\ep}(\beta,\delta)Cov_{n,\ep}(\beta,\delta).$$
Hence
$$\lim_{n\to \infty}\frac{1}{n}\log Cov_{n,\ep}(\beta,\delta)=\inf_n \frac{1}{n}\log Cov_{n,\ep}(\beta,\delta).$$
It is easy to see that 
$$FKPOsr_{m,2\ep;\beta,\delta}\leq Cov_{m,\ep}(\beta,\delta).$$
Hence \begin{equation*}
\begin{split}
\varlimsup_{\delta\to0}\inf_{\alpha>0}\varlimsup_{m\to\infty}\frac{1}{m}\log FKPOsr_{m,2\ep;\alpha,\delta}
&\leq \varlimsup_{\delta\to0}\inf_{\alpha>0}\varlimsup_{m\to\infty}\frac{1}{m}\log Cov_{m,\ep}(\alpha,\delta)\\
&\leq\varlimsup_{\delta\to0}\inf_{\alpha>0}\inf_{m}\frac{1}{m}\log Cov_{m,\ep}(\alpha,\delta)\\
&\leq \varlimsup_{\delta\to 0}\varlimsup_{m\to \infty}\frac{1}{m}\log sp_{m,\frac{\delta}{4}}+L(\ep)\\
&=P(X,T,f)+L(\ep).
\end{split}
\end{equation*}
Let $\ep\to 0,$ we have $PFKPO(X,T,f)\leq P(X,T,f).$
	\end{proof}
\begin{rem}\begin{enumerate}	
		\item It is easy to see that
		\begin{equation*}
		\begin{split}
	 P(X,T,f)&\leq\lim_{\ep\to0}\inf_{\delta>0}\inf_{\alpha> 0}\varlimsup\limits_{n\to\infty}\frac{1}{n}\log FKPOsp_{n,\ep;\alpha,\delta}\\
		&\leq\lim_{\ep\to0}\varlimsup_{\delta\to0}\inf_{\alpha> 0}\varlimsup\limits_{n\to\infty}\frac{1}{n}\log FKPOsp_{n,\ep;\alpha,\delta},
		\end{split}
		\end{equation*}
		we can easily establish a similar result if we replace ``$\varlimsup\limits_{\delta\to0}$"  by ``$\inf\limits_{\delta> 0}$" in the definition of $PFKPO(X,T,f).$
			\item The set of periodic Feldman-Katok pseudo-orbits contains the set of periodic pseudo-orbits and is contained in the set of Feldman-Katok pseudo-orbits, hence  the result with respect to periodic Feldman-Katok pseudo-orbits can be directly obtained. 
	\end{enumerate}
	
\end{rem}

\subsection{Pseudo-orbits and scaled pressure}

In \cite{cl}, the authors studied scaled pressure with respect to pseudo-orbits. In this subsection we show  that in \cite[Theorem F]{cl}, the condition that $T$ is Lipschitz continuous is redundant.

\medskip

First we recall some concepts.

A function $S:(0,1)\to (0,\infty)$ is called a {\it scale function} if for any $\lambda\in (0,\infty), $  $$\lim\limits_{x\to 0}\frac{S(\lambda x)}{S(x)}=1.$$

Let  $f\in C(X)$ and $S$ be a scale function, define $$Spa_{n,\ep}(X,T,f,S)=\inf\{\sum_{x\in F}e^{S(\ep)\sum_{i=0}^{n-1}f(T^ix)}: F \text{ is\ an }(n,\ep) \text{-spanning\ set \ of } X\},$$
$$Sep_{n,\ep}(X,T,f,S)=\sup\{\sum_{x\in E}e^{S(\ep)\sum_{i=0}^{n-1}f(T^ix)}: E \text{ is\ an }(n,\ep) \text{-separated\ set \ of } X\}.$$ 
{\it Scaled pressure of $(X,T)$ with respect to $f$ and $S$} is defined by
\begin{equation*}
	\begin{split}
		Sdim(X,T,f,S)&=\varlimsup\limits_{\ep\to0}\frac{\varlimsup\limits_{n\to\infty}\frac{1}{n}\log Spa_{n,\ep}}{S(\ep)}\\
		&=\varlimsup\limits_{\ep\to0}\frac{\varlimsup\limits_{n\to\infty}\frac{1}{n}\log Sep_{n,\ep}}{S(\ep)}.	
	\end{split}
\end{equation*}
Denote by $$PPO_{n,\a}(X,T)=\{{\bf x}\in PO_\a(X,T): x_{n+k}=x_k, k\in\N\}$$ the $n$-period $\a$ pseudo-orbits. 

Define \begin{equation*}
	\begin{split}
		&POSpa_{n,\ep,\a}(X,T,f,S)(resp.\ PPOSpa_{n,\ep,\a}(X,T,f,S))=\\
		&\inf\{\sum_{{\bf x}\in F}e^{S(\ep)\sum_{i=0}^{n-1}f(x_i)}: F \text{ is\ an }(n,\ep) \text{-spanning\ set \ of } PO_\a(resp.\ PPO_{n,\a})\},
	\end{split}
\end{equation*}
\begin{equation*}
	\begin{split}
		&POSep_{n,\ep,\a}(X,T,f,S)(resp.\ PPOSep_{n,\ep,\a}(X,T,f,S))=\\
		&\sup\{\sum_{{\bf x}\in E}e^{S(\ep)\sum_{i=0}^{n-1}f(x_i)}: E \text{ is\ an }(n,\ep) \text{-separated\ set \ of } PO_\a(resp.\ PPO_{n,\a})\}.
	\end{split}
\end{equation*}

Define {\it scaled pressure with respect to pseudo-orbits (resp. periodic  pseudo-orbits)} by
\begin{equation*}
	\begin{split}
	&PSdim(X,T,f,S)(resp.\ PPSdim(X,T,f,S))\\
		&=\varlimsup\limits_{\ep\to0}\frac{\lim\limits_{\a \to 0}\varlimsup\limits_{n\to\infty}\frac{1}{n}\log POSpa_{n,\ep,\a}(resp.\ PPOSpa_{n,\ep,\a})}{S(\ep)}\\
		&=\varlimsup\limits_{\ep\to0}\frac{\lim\limits_{\a \to 0}\varlimsup\limits_{n\to\infty}\frac{1}{n}\log POSep_{n,\ep,\a}(resp.\ PPOSep_{n,\ep,\a})}{S(\ep)}.	
	\end{split}
\end{equation*}

\medskip

Now we show  that in \cite[Theorem F]{cl}, the condition that $T$ is Lipschitz continuous is redundant.
\begin{thm}
	Let $(X,T)$ be a TDS, $f\in C(X)$ and $S$ be a scale function. Then $$Sdim(X,T,f,S)=PSdim(X,T,f,S).$$
	If in addition $f\geq 0$ and $S$ is non-increasing, then $$Sdim(X,T,f,S)=PPSdim(X,T,f,S).$$
\end{thm}
\begin{proof}
By definition we have $Sdim(X,T,f,S)\leq PSdim(X,T,f,S).$ Now we prove the opposite direction.

Fix $\ep>0.$	Let $X=\bigcup_{i=1}^{l} X_i$ with $diam X_i<\frac{\ep}{4},i=1,\ldots,l.$ Choose some $z_i\in X_i$.

For $n>\frac{2diam X}{\ep}$ and $m$, let $F$ be an $(n+m,\frac{\ep}{4})$-spanning set of $X$. We can find $$0<\delta_0<\delta_1<\ldots<\delta_{2(n+m)}=\frac{\ep}{4}$$ such that for every $\delta_k, k=0,1,\ldots,2(n+m)-1$,
$$d(x,y)<\delta_k\Rightarrow d(T^ix,T^iy)<\delta_{k+1}, i=0,\ldots,2(n+m)-1.$$
Let $$\a<\min\{\delta_0, \delta_{k+1}-\delta_k:k=0,1,\ldots,2(n+m)-1 \}.$$ 
Let  	
$$\tilde{F}=\{(x,Tx,\ldots,T^{i}x,\ldots, T^{m+n-1}x,z_1,z_1,\ldots): x\in F\}. $$
We have  $\bigcup_{{\bf y}\in \tilde{F}}B_{\tilde{d}_m}({\bf y},\ep)$ covers $PO_\a$.
Denote
\begin{equation*}
	\begin{split}
		&Cov_{n,\ep}(\beta)=\\
		&\inf\{\sum_{A\in \mathcal{A}}\sup_{{\bf x}\in A}e^{S(\ep)\sum_{i=0}^{n-1}f(x_i)}:\mathcal{A}\text{ is\ a\ finite\ subcover\ of\ } \vee_{i=0}^{n-1}\sigma^{-i}\mathcal{A}_\ep \text{ covers } PO_{\beta} \},
	\end{split}
\end{equation*}
where  $\A_\ep$ denote the  cover of $(X^\N,\tilde{d})$  consists of all open $\ep$-balls.  
We have $$Cov_{m,\ep}(\alpha)\leq \sum_{{\bf y}\in \tilde{F}}\sup_{{\bf z}\in B_{\tilde{d}_m}({\bf y},\ep)}e^{S(\ep)\sum_{i=0}^{m-1}f(z_i)}.$$
Here $\alpha$ dependents on $m$.

For $x\in F, {\bf y}=(x,Tx,\ldots,T^{i}x,\ldots, T^{m+n-1}x,z_1,z_1,\ldots)\in \tilde{F}$ and ${\bf z}\in B_{\tilde{d}_m}({\bf y},\ep)$, we have $d(T^ix,z_i)<\ep,i=0,\ldots,m-1.$
It follows that $$\sum_{i=0}^{m-1}f(z_i)-\sum_{i=0}^{m-1}f(T^ix)\leq mL(\ep),$$
where $$L(\ep)=\sup_{|x-y|\leq\ep}|f(x)-f(y)|.$$ Hence
\begin{equation*}
	\begin{split}
		\sum_{{\bf y}\in \tilde{F}}\sup_{{\bf z}\in B_{\tilde{d}_m}({\bf y},\ep)}e^{S(\ep)\sum_{i=0}^{m-1}f(z_i)}
		\leq \sum_{x\in F}e^{S(\ep)\sum_{i=0}^{m-1}f(T^ix)}e^{S(\ep)mL(\ep)}.
	\end{split}
\end{equation*}
We have
\begin{equation*}
	\begin{split}
		Cov_{m,\ep}(\alpha)&\leq e^{S(\ep)mL(\ep)}\sum_{x\in F}e^{S(\ep)\sum_{i=0}^{m-1}f(T^ix)}\\
		&\leq e^{S(\ep)(mL(\ep)+n||f||)}\sum_{x\in F}e^{S(\ep)\sum_{i=0}^{m+n-1}f(T^ix)}.
	\end{split}
\end{equation*}
Hence 
\begin{equation*}
	\begin{split}
		\inf_m \inf_{\alpha>0}\frac{1}{m}\log Cov_{m,\ep}(\alpha)
		\leq S(\ep)L(\ep)+
		\varlimsup_{m\to \infty}\frac{1}{m}\log Spa_{m,\frac{\ep}{4}}.
	\end{split}
\end{equation*}
It is easy to see that 
$$\lim_{n\to \infty}\frac{1}{n}\log Cov_{n,\ep}(\beta)=\inf_n \frac{1}{n}\log Cov_{n,\ep}(\beta)$$
and 
$$POSep_{m,2\ep,\beta}\leq Cov_{m,\ep}(\beta).$$
Hence \begin{equation*}
	\begin{split}
		\inf_{\alpha>0}\varlimsup_{m\to\infty}\frac{1}{m}\log POSep_{m,2\ep,\a}
		&\leq \inf_{\alpha>0}\varlimsup_{m\to\infty}\frac{1}{m}\log Cov_{m,\ep}(\alpha)\\
		&\leq\inf_{\alpha>0}\inf_{m}\frac{1}{m}\log Cov_{m,\ep}(\alpha)\\
		&\leq \varlimsup_{m\to \infty}\frac{1}{m}\log Spa_{m,\frac{\ep}{4}}+S(\ep)L(\ep).
	\end{split}
\end{equation*}
 It follows that 
\begin{equation*}
	\frac{\inf\limits_{\alpha>0}\varlimsup\limits_{m\to\infty}\frac{1}{m}\log POSep_{m,2\ep,\a}}{S(2\ep)}	
		\leq \frac{S(\frac{\ep}{4})}{S(2\ep)}\frac{\varlimsup\limits_{m\to \infty}\frac{1}{m}\log Spa_{m,\frac{\ep}{4}}}{S(\frac{\ep}{4})}+\frac{S(\ep)}{S(2\ep)}L(\ep).
\end{equation*}
Let $\ep\to 0,$ we have $Sdim(X,T,f,S)\geq PSdim(X,T,f,S).$	

If in addition $f\geq 0$ and $S$ is non-increasing, then $$Sdim(X,T,f,S)=PPSdim(X,T,f,S)$$
just follows from the discussion in \cite[Theorem F]{cl}.
\end{proof}

\section{Topological pressure for Feldman-Katok metric }
In this section we study topological pressure for Feldman-Katok metric.
 
 \medskip
 
	First we recall the definition of Feldman-Katok metric.

Let $(X,T)$ be a TDS, $n\in\N$ and $\delta>0$. For $x, y\in X$, we define an  $(n, \delta)$-match of $x$ and $y$ to be an  order
	preserving (i.e. $\pi(i)<\pi(j)$ whenever $i<j$) bijection $\pi: D(\pi) \rightarrow R(\pi)$ such that $D(\pi),R(\pi) \subset\{0, 1,\ldots, n-1\}$
and for every $i \in D(\pi)$ we have $d(T^ix, T^{\pi(i)}y) < \delta$. Let $|\pi|$ be the cardinality of $D(\pi)$. We set$$\bar{f}_{n,\delta}(x, y) = 1-\frac{\max\{|\pi| : \pi \text{ is an }(n, \delta)\text{-match of }x \text{ and } y\}}{n}$$
	Define the Feldman-Katok metric  by  $$d_{FK_n}(x, y) = \inf\{\delta> 0: \bar{f}_{n,\delta}(x, y) < \delta\}.$$

	For $n\in\N$ and $\ep>0$,
We say a subset $E\subset X$ is a $FK$-$(n,\ep)$-{\it spanning set of $X$}  if for any $x\in X$, there exists $y\in E$ with $d_{FK_n}(x,y)\leq \ep.$ 
We say	a subset $F\subset X$ is a {\it  $FK$-$(n,\ep)$-separated set of $X$} if for $x\neq y\in F,$ $d_{FK_n}(x,y)>\ep.$

For $f\in C(X)$, define
	$$FKsp_{n,\ep}(X,T,f)=\inf\{\sum_{x\in F}e^{\sum_{i=0}^{n-1}f(T^ix)}:  F\text{\ is\ a\ }FK\text{-}(n,\ep) \text{-spanning\ set\ of}\ X\},$$
	$$FKsr_{n,\ep}(X,T,f)=\sup\{\sum_{x\in E}e^{\sum_{i=0}^{n-1}f(T^ix)}: E\text{\ is\ a\ }FK\text{-}(n,\ep) \text{-separated\ set\ of}\ X\}.$$

Define 	the {\it topological  pressure  with respect to Feldman-Katok metric} by
$$PFK(X,T,f)=\lim_{\ep\to 0}\varlimsup\limits_{n\to\infty}\frac{1}{n}\log FKsp_{n,\ep}(X,T,f).$$

Similar as Lemma \ref{lem}, we can use separated set in the definition above.
\begin{lem}
	Let $(X,T)$ be a TDS and $f\in C(X).$ Then
\begin{equation*}
\begin{split}
PFK(X,T,f)=\lim_{\ep\to 0}\varlimsup\limits_{n\to\infty}\frac{1}{n}\log FKsr_{n,\ep}(X,T,f).
\end{split}
\end{equation*}	
\end{lem}

Now we prove the topological  pressure  with respect to Feldman-Katok metric is equal to the classical topological pressure.
\begin{thm} Let $(X,T)$ be a TDS and $f\in C(X)$. 
Then $$PFK(X,T,f)=P(X,T,f).$$
\end{thm}
\begin{proof}
	It is easy to see that $PFK(X,T,f)\leq P(X,T,f)$, we prove the opposite direction.
	
Fix $\ep>0.$ Let $X=\cup_{i=1}^{l} X_i$ with $diam X_i<\ep, i=1,\ldots,l.$ For $\delta\leq\ep$,	let $F$ be a $FK$-$(n,\frac{\delta}{2})$-spanning set of $X$. Let $x\in F$ and consider sets with following form
$$A_{0}\cap T^{-1}A_1\cap \ldots\cap T^{-(j_1-1)}A_{j_1-1}\cap T^{-j_1}B_\frac{\delta}{2}(T^{i_1}x)\cap  T^{-(j_1+1)}A_{j_1+1}\cap\ldots
$$ $$
\cap T^{-(j_k-1)}A_{j_k-1}\cap T^{-j_k}B_\frac{\delta}{2}(T^{i_k}x)\cap T^{-(j_k+1)}A_{j_k+1}\ldots\cap T^{-(n-1)}A_{n-1},$$

where $$(1-\frac{\delta}{2})n<k\leq n,$$ $$0\leq i_1<\ldots<i_k\leq n-1,$$ $$0\leq j_1<\ldots<j_k\leq n-1$$ and 
$A_i\in\{X_1,\ldots,X_{l}\}.$

 For every non-empty set with the form above, choose  some point $y$ in it.
 Let $F^\prime_x$ be the set consists of such $y$ and let $F^\prime=\cup_{x\in F}F^\prime_x.$ We will prove 
 $F^\prime$ is an $(n,\ep)$-spanning set.
 For every $z\in X$, since $F$ is a $FK$-$(n,\frac{\delta}{2})$-spanning set, there is some $x\in F$ with $d_{FK_n}(x,z)<\frac{\delta}{2}$. Hence there exist $$(1-\frac{\delta}{2})n<k\leq n,$$ $$0\leq i_1<\ldots<i_{k}\leq n-1,$$ $$0\leq j_1<\ldots<j_{k}\leq n-1$$ such that $$d(T^{i_h}x,T^{j_h}z)<\frac{\delta}{2}, h=1,\ldots,k.$$ Assume $T^iz\in X_{t_i},i=0,\ldots,n-1.$ Then 
 $$X_{t_0}\cap \ldots\cap T^{-(j_1-1)}X_{t_{j_1-1}}\cap T^{-j_1}B_\frac{\delta}{2}(T^{i_1}x)\cap  T^{-(j_1+1)}X_{t_{j_1+1}}\ldots\cap T^{-(n-1)}X_{t_{n-1}}\neq\emptyset$$
 since $z$ is in it. By the construction of $F^\prime$, there is $y\in F^\prime$  in the set above. It is easy to see that $d_n(y,z)<\ep.$  Hence $F^\prime$ is an $(n,\ep)$-spanning set.

We have
 \begin{equation*}
 \begin{split}
 	sp_{n,\ep}(X,T,f)&\leq\sum_{y\in F^\prime}e^{\sum_{i=0}^{n-1}f(T^iy)} \leq\sum_{x\in F}\sum_{y\in F^\prime_x}e^{\sum_{i=0}^{n-1}f(T^iy)}\\
 	&\leq\sum_{x\in F}\sum_{y\in F^\prime_x}e^{\sum_{i=0}^{n-1}f(T^ix)+\sum_{i=0}^{n-1}f(T^iy)-\sum_{i=0}^{n-1}f(T^ix)}\\
 		&\leq\sum_{x\in F}\sum_{y\in F^\prime_x}e^{\sum_{i=0}^{n-1}f(T^ix)}e^{nL(\frac{\delta}{2})+n\delta||f||},
 \end{split}
 \end{equation*}
where $L(\ep)=\sup\limits_{|x-y|\leq \ep}|f(x)-f(y)|.$
 Since \begin{equation*}
 \begin{split}
 |F^\prime_x|\leq \sum_{(1-\frac{\delta}{2})n<k\leq n}(C^k_n)^2l^{n-k}\leq\frac{\delta}{2}n(C^{[\frac{\delta}{2}n]}_n)^2l^{\frac{\delta}{2}n},
 \end{split}
 \end{equation*} we have 
 $$sp_{n,\ep}\leq e^{(nL(\frac{\delta}{2})+n\delta||f||)}\frac{\delta}{2}n(C^{[\frac{\delta}{2}n]}_n)^2l^{\frac{\delta}{2}n}\sum_{x\in F}e^{\sum_{i=0}^{n-1}f(T^ix)}
 .$$
 Hence we have 
 $$sp_{n,\ep}\leq e^{(nL(\frac{\delta}{2})+n\delta||f||)}\frac{\delta}{2}n(C^{[\frac{\delta}{2}n]}_n)^2l^{\frac{\delta}{2}n}FKsp_{n,\frac{\delta}{2}}.
 $$
 It follows that
 \begin{equation*}
 \begin{split}
 \varlimsup\limits_{n\to\infty}\frac{1}{n}\log sp_{n,\ep}\leq L(\frac{\delta}{2})+\delta||f||
 +\frac{\delta}{2}\log l+\lim\limits_{n\to \infty}\frac{1}{n}2\log C_n^{[\frac{\delta}{2}n]}
 +\varlimsup\limits_{n\to\infty}\frac{1}{n}\log FKsp_{n,\frac{\delta}{2}}.
 \end{split}
 \end{equation*}
 Let $\delta\to 0$, we have $$\varlimsup\limits_{n\to\infty}\frac{1}{n}\log sp_{n,\ep}\leq PFK(X,T,f).$$
Let $\ep\to 0$, we get $P(X,T,f)\leq PFK(X,T,f).$
	\end{proof}
\begin{rem}
Both 	Feldman-Katok pseudo-orbits and Feldman-Katok metric allow appropriate time delay and error when computing topological pressure. Feldman-Katok pseudo-orbits contain pseudo-orbits hence contain real orbits, it approximate topological pressure from upper bound. On the other hand, Feldman-Katok metric is less than Bowen metric, it approximate topological pressure from lower bound.
	\end{rem}


\begin{thebibliography}{14}
\bibitem{entro}	
 R. Adler, A. Konheim, M. McAndrew, 
 \newblock Topological entropy.
  \newblock {\it Trans. Amer. Math. Soc.}, 114, (1965), 309–319.
 
 \bibitem{entro1}
  E. Dinaburg,
\newblock   On the relations among various entropy characteristic of dynamical systems.
\newblock {\it Izv. Akad. Nauk SSSR, Ser. Mat.}, 35, (1971), 324–366.

\bibitem{entro2}
 R. Bowen, 
\newblock Entropy for group endomorphisms and homogeneous spaces.
\newblock {\it Trans. Amer. Math. Soc.}, 153, (1971), 401–414.

\bibitem{6}
 R. Bowen,  Topological entropy for noncompact sets. {\it Trans. Amer. Math. Soc.}, 184, (1973), 125–136. 
 
\bibitem{bo}
 R. Bowen, Equilibrium States and the Ergodic Theory of Axiom A Diffeomorphisms. {\it Lecture Notes in Math.,} vol. 470,
Springer-Verlag, New York, 1975.

\bibitem{pre1}
R. Bowen, Hausdorff dimension of quasicircles. {\it Publ. Math. IHÉS.}, 50, (1979), 11–25.

\bibitem{co}
 C. Conley, Isolated Invariant Sets and the Morse Index. {\it CBMS Conference Series,} vol. 38, 1978.

\bibitem{pse1} M. Misiurewicz, Remark on the definition of topological entropy. in: {\it Dynamical Systems and Partial Differential Equations},
1986, pp. 65–68.

\bibitem{pse2} M. Barge, R. Swanson, Pseudo-orbits and topological entropy. {\it Proc. Amer. Math. Soc.}, 109, (1990), 559–566.

\bibitem{pse3} M. Hurley, On topological entropy of maps. {\it Ergod. Theory Dyn. Syst.}, 15, (1995), 557–568.

\bibitem{yz} K. Yan, F. Zeng,  Topological entropy, pseudo-orbits and uniform spaces. {\it Topology Appl.} 210, (2016), 168–182.

\bibitem{fk1}
J. Feldman, New K-automorphisms and a problem of Kakutani. {\it Israel J. Math.,} 24 (1), 16–38, 1976.

\bibitem{fk2}
 A. Katok, Monotone equivalence in ergodic theory. {\it Izv. Akad. Nauk SSSR Ser. Mat.,} 41 (1), 104–157,
231, 1977.

\bibitem{fk3}
D. Ornstein, D. Rudolph, B. Weiss, Equivalence of measure preserving transformations. {\it Mem. Amer.
Math. Soc.,} 37, (262), 1982.

\bibitem{fk4}
T. Downarowicz, D. Kwietniak, M.\L acka,  Uniform continuity of entropy rate with respect to the $\bar{f}$-pseudometric. {\it IEEE Trans. Inform. Theory}, 67, (2021), no. 11, 7010–7018.

\bibitem{cai}
F. Cai, J. Li, On Feldman-Katok metric and entropy formulae. {\it arXiv:} 2104.12104.



\bibitem{pre2} Y. Pesin, Dimension Theory in Dynamical Systems. {\it Contemporary Views and Applications,} Chicago Lectures in Mathematics, University of Chicago Press, Chicago, IL, 1997.

	\bibitem{31}
Y. Pesin, B. Pitskel,  Topological pressure and the variational principle for noncompact sets. {\it Funct.
Anal. Appl.}, 18(4), (1984), 307–318. 

\bibitem{pre3} D. Ruelle, Thermodynamic Formalism. {\it The Mathematical Structures of Classical Equilibrium Statistical Mechanics, Encyclopedia of Mathematics and Its Applications,} vol.5, Addison–Wesley Publishing Co., Reading, Mass., 1978.

\bibitem{pre4} D. Ruelle, Repellers for real analytic maps. {\it Ergod. Theory Dyn. Syst.}, 2, (1982), 99–107.

%\bibitem{gro}
% M. Gromov, 
% \newblock Topological invariants of dynamical systems and spaces of holomorphic maps.
% \newblock I, Math. Phys. Anal. Geom. 2 (4) (1999) 323–415.

%\bibitem{lin1}	
%   E. Lindenstrauss, M. Tsukamoto, 
%   \newblock From rate distortion theory to metric mean dimension: variational principle, 
%   \newblock IEEE
%Trans. Inf. Theory 64 (5) (2018) 3590–3609.

%\bibitem{lin2}
%   E. Lindenstrauss, B. Weiss,
% \newblock  Mean topological dimension, 
% \newblock Isr. J. Math. 115 (2000) 1–24.	
 
% \bibitem{lin3}
% E. Lindenstrauss and M. Tsukamoto. From rate distortion theory to metric mean dimension: variational principle. IEEE
% Trans.Inform. Theory 64 (2018), 3590-3609.
 
 \bibitem{ru}
D. Ruelle,
\newblock Statistical mechanics on a compact set with $Z^v$ action satisfying expansiveness and specification.
\newblock {\it Trans. Amer. Math. Soc.}, 187, (1973), 237–251.

\bibitem{wal}
 P. Walters, A variational principle for the pressure of continuous transformations. {\it  Am. J. Math.}, 97, (1975), 937–971.

\bibitem{cl}
 D. Cheng, Z. Li, 
\newblock  Scaled pressure of dynamical systems. 
\newblock {\it J. Differential Equations.}, 342, (2023), 441–471.

\bibitem{kl}
 D. Kwietniak, M. \L acka,
 \newblock Feldman-Katok pseudometric and
the GIKN construction of nonhyperbolic ergodic measures.
\newblock
{\it arXiv:} 1702.01962.

	\bibitem{gk}
	F. García-Ramos, D. Kwietniak,
 \newblock On topological models of zero entropy loosely Bernoulli systems.  \newblock {\it Trans. Amer. Math. Soc.}, 375, (2022), no. 9, 6155–6178. 

%\bibitem{xcy}
%Y.Xie, E.Chen, R.Yang,
% \newblock Feldman-Katok metric mean dimension,
% \newblock arXiv:2208.09645.
 


\end{thebibliography}
\end{document}